\documentclass[10pt]{amsart}
\usepackage{graphicx}                                    
\usepackage{amssymb}
\usepackage{enumitem}
\usepackage{color}
\newtheorem{theorem}{Theorem}[section] 
\newtheorem{lemma}[theorem]{Lemma}
\newtheorem{conjecture*}[theorem]{Conjecture}
\newtheorem{proposition}[theorem]{Proposition}

\newtheoremstyle{notauto}{}{}{\itshape}{}{\bfseries}{.}{0.5em}{\thmnote{#3}}
\theoremstyle{notauto}

\theoremstyle{definition}
\newtheorem{definition}[theorem]{Definition}
\newtheorem{example}[theorem]{Example}
\theoremstyle{remark}
\newtheorem{remark}[theorem]{Remark}

\newcommand{\Qdp}{\text{Qd}(p)}

\newcommand{\Qd}{\text{Qd}}

\newcommand{\bF}{{\mathbb F}}
\newcommand{\bZ}{{\mathbb Z}}

\newcommand{\cA}{{\mathcal A}}

\newcommand{\res}{\text{Res}}
\newcommand{\rk}{\text{rk}}
\newcommand{\ind}{\text{Ind}}

\begin{document}
\baselineskip13pt

\title[Finite groups of rank two which do not  involve $\Qdp$]{Finite groups of rank two which do \\ not involve $\Qdp$}

\author{Muhammet Yas{\. i}r K{\i}zmaz}
\author{Erg{\" u}n Yal{\c c}{\i}n}
\address{Department of Mathematics, Bilkent University, 06800 
Bilkent, Ankara, Turkey}

\email{yasirkizmaz@bilkent.edu.tr, yalcine@fen.bilkent.edu.tr}

\begin{abstract}  Let $p>3$ be a prime. We show that if $G$ is a finite group with $p$-rank equal to 2, then 
$G$ involves $\Qdp$ if and only if $G$ $p'$-involves $\Qdp$.  This allows us to use a version of Glauberman's ZJ-theorem to
give a more direct construction of finite group actions on mod-$p$ homotopy spheres. 
We give an example to illustrate that the above conclusion 
does not hold for $p \leq 3$.  
\end{abstract}
	
\keywords{}

\thanks{2010 {\it Mathematics Subject Classification.} Primary:
20Dxx; Secondary: 20E25, 57S17.}
 

\maketitle

\section{Introduction}
	
Throughout the paper all groups are finite, and $p$ denotes a prime number. A $p$-group is said to be of \emph{rank $k$} if the largest possible order of an elementary abelian subgroup of the group is $p^k$. We say that a group $G$ is of \emph{$p$-rank $k$} if a Sylow $p$-subgroup of $G$ is of rank $k$. We denote the $p$-rank of $G$ by $\rk_p(G)$.

The group $\Qdp$ is defined to be the semidirect product $$\Qdp:=(\bZ/p \times \bZ / p )\rtimes SL(2,p)$$ where the action of the group $SL(2,p)$ on $\bZ /p \times \bZ /p$ is the usual action of $2\times 2$ matrices on a two-dimensional vector space. A group $G$ is said to \emph{involve $\Qdp$}  if there exist subgroups $K \lhd H \leq G$ such that $H/K \cong \Qdp$. We say $G$ \emph{$p'$-involves $\Qdp$} if there exist  $K \lhd H \leq G$ such that $K$ has order coprime to $p$ and $H/K \cong \Qdp$. If a group $p'$-involves $\Qdp$, then obviously it involves $\Qdp$, but the converse does not hold in general. We show that for finite groups with $p$-rank equal to $2$ these two conditions are equivalent when $p>3$.  
  
\begin{theorem}\label{thm:intro1}
Let $G$ be a finite group with $rk_p(G)=2$, where $p>3$ is a prime. Then $G$ involves $\Qdp$ if and only if $G$ $p'$-involves $\Qdp$.	
\end{theorem}
	
Theorem \ref{thm:intro1} is proved in Section \ref{sect:Proof}. The key step is the case where $G$ involves $\Qdp$ with subgroups $K\lhd H\leq G$ where  $K$ is a $p$-group. This case is handled by using the classification of $p$-groups of rank 2. 	
Theorem \ref{thm:intro1} is no longer true when $p=3$. We illustrate this by constructing a group extension of $\Qd(3)$ by a cyclic group of order $3$ in a way that the extension group does not $3'$-involve $\Qd(3)$ (see Example \ref{ex:p=3}). For $p=2$, there is a similar example (see Example \ref{ex:p=2}). 
 
The condition that $G$ does not involve $\Qdp$ appears in Glauberman's ZJ-theorem  \cite[Thm B]{Glauberman1}. Let $G$ be a finite group, and $S$ be a Sylow $p$-subgroup of $G$. The Thompson subgroup $J(S)$ of $S$ is defined to be the subgroup generated by all abelian $p$-subgroups of $S$ with maximal 
order. The center of $J(S)$ is denoted by $ZJ(S)$ and its normalizer in $G$ by $N_G(ZJ(S))$. Given a subgroup $H$ of $G$ containing $S$, we say \emph{$H$ controls $G$-fusion in $S$} if for every $P\leq S$ and $g \in G$ such that $gPg^{-1}\leq S$, there exist $h\in H$ and $c\in C_G(P)$ such that $g=hc$. Glauberman's $ZJ$-theorem \cite[Thm B]{Glauberman1} states that if $G$ is a finite group that does not involve $\Qdp$, and $S$ is a Sylow $p$-subgroup of $G$, where $p$ is odd, then $N_G(ZJ(S))$ controls $G$-fusion in $S$. There is a version of Glauberman's ZJ-theorem due to Stellmacher which also works for $p=2$. We find Stellmacher's version of Glauberman's theorem more useful for our purpose even for the $p$ odd case (see Theorem \ref{thm:SZJ}).

The condition that $G$ $p'$-involves $\Qdp$ appears in the construction of finite group actions on products of spheres, 
particularly in the construction of mod-$p$ spherical fibrations over the classifying space $BG$ of a finite group $G$.  
One requires these spherical fibrations to have a $p$-effective Euler class (see Section \ref{sect:Glauberman} for definitions). 
Jackson \cite{Jackson} showed that if $G$ is a finite group of rank 2 that does not $p'$-involve $\Qdp$ for any odd prime $p$, 
then there is a spherical fibration over $BG$ with an effective Euler class.  Jackson proves this theorem using results from 
two papers, one on the homotopy theory of maps between classifying spaces \cite{JacksonQuotient} and the other 
on representations that respect fusion \cite{Jackson}.

The motivation for proving Theorem \ref{thm:intro1} comes from the question of whether or not the constructions of mod-$p$ spherical fibrations are related to Glauberman's ZJ-theorem. Having shown that these two theorems have equivalent conditions at least for $p>3$,  we consider whether Glauberman's ZJ-theorem can be used directly to obtain a mod-$p$ spherical fibration with a $p$-effective Euler class, providing an alternative to Jackson's construction. We prove this for a finite group of arbitrary rank and for any prime $p$, using Stellmacher's version of Glauberman's ZJ-theorem. 

\begin{theorem}\label{thm:intro2}
Let $G$ be a finite group that does not involve $\Qdp$, where $p$ is a prime. Then 
there is a mod-$p$ spherical fibration over $BG$ with a $p$-effective Euler class.	
\end{theorem}

Theorem \ref{thm:intro2} is proved in Section \ref{sect:Glauberman}.  In the proof we use Stellmacher's theorem (Theorem \ref{thm:SZJ}) to obtain that there is a homotopy equivalence $BG_p ^\wedge \simeq BN_p ^\wedge$ between $p$-completions of classifying spaces of $G$ and $N=N_G (W(S))$. We then show that over $BN$ the desired spherical fibration can be constructed using the Borel construction $EN\times _N S(V) \to BN$ for a linear sphere $S(V)$.   

Note that Theorem \ref{thm:intro1} and Theorem \ref{thm:intro2} together give a different proof for Jackson's theorem for $p>3$. We believe Theorem \ref{thm:intro2} is interesting in its own right for constructing actions on products of spheres for finite groups with arbitrary rank.

\vskip 5pt

\noindent 
{\bf Acknowledgement:}  Both authors are supported by a T\" ubitak 1001 project (grant no. 116F194).  We thank the referee for a careful reading of the paper and many useful comments on the paper. In particular the proof of Theorem \ref{thm:intro1} is significantly improved by referee's comments.


\section{The proof of Theorem \ref{thm:intro1}}\label{sect:Proof}

We first consider the case where $G$ involves $\Qdp$ with $U\lhd G$ where  $U$ is a $p$-group.  
		
\begin{lemma}\label{main lemma}
Let $U\lhd G$. Suppose that $U$ is a $p$-subgroup of $G$ such that $G/U\cong \Qdp$.  
If $rk_p(G)=2$ and $p>3$ then $U=1$, that is, $G\cong\Qdp$.
\end{lemma}
	
\begin{proof}
Let $P\in Syl_p(G).$ Note that $P$ is of rank $2$ and $P/U$ is isomorphic to a Sylow $p$-subgroup 
of $\Qdp$ by the hypothesis, and so $P/U$ is isomorphic to an extra-special $p$-group of order $p^3$ 
and of exponent $p$.  Note that $P$ belongs to one of the $3$-possible family of $p$-groups by 
the classification of $p$-groups of rank $2$ (see \cite[Thm A.1]{DRV}).

(i) $P$ is a metacyclic group. It follows that $P/U$ is also a metacyclic group of order $p^3$. 
However this is not possible as the exponent of $P/U$ is $p$. Thus, $P$ cannot be metacyclic.
			
(ii) $P=C(p,r):=\langle a, b, c \mid a^p = b^p = c^{p^{r-2}} = 1, [a, b] = c^{p^{r-3}}, [a, c] = [b, c] = 1 \rangle$ 
where $r\geq 3$. First suppose that $r>3.$ Since $P/U$ has exponent $p$, we have $[a,b]=c^{p^{r-3}}\in U$. 
It follows that $P/U$ is abelian as $[a,c]=[b,c]=1$. This contradiction shows that $r=3$. In this case, one can 
easily see that $|P|=p^3=|P/U|$, and so $U=1$.
			
(iii) $P= G(p, r; \epsilon):=\langle a, b, c \mid a^p=b^p=c^{p^{r-2}}=[b, c]=1, [a, b^{-1}]=c^{\epsilon p^{r-3}}, [a, c]=b \rangle$ 
where $r\geq 4$ and $\epsilon$ is either $1$ or a quadratic non-residue at modulo $p$.
Since $[a,c]=b$ and $[b,c]=1$, we have $[a,c^p]=[a,c]^p=b^p=1$. Thus, $c^p\in Z(P)$. Moreover, $c^p\in U$ since
$P/U$ is of exponent $p$. It is easy to see that $P/\langle c^p \rangle$ has order $p^3$ by the given presentation. 
This forces that $U=\langle c^p \rangle$, and so $U\leq Z(P)$. Clearly, we have an embedding of $G/C_G(U)$ into 
the $Aut(U)$. Note that $G/C_G(U)\cong (G/U)/(C_G(U)/U)$ and $Aut(U)$ is cyclic, and so $G/C_G(U)$ is isomorphic 
to a cyclic quotient of $\Qdp$. As a consequence, $C_G(U)=G$, that is, $U\leq Z(G)$.
			
Let $A$ be a normal subgroup of $G$ such that $G/A \cong SL(2,p)$ and $A/U\cong \bZ /p \times \bZ /p $. Note that 
$A$ is a maximal subgroup in $P$. Since the Frattini subgroup of $P$ is $\langle b, c^p\rangle$,
either $A=\langle b, c \rangle$, or $A=\langle b, c^p, ac^i\rangle$ for some $i=0, \dots, p-1$. If $A=\langle b, c\rangle $, 
then $A$ is abelian, hence $A\cong C_p \times C_{p^{r-2}}$. If $A=\langle b, c^p, a \rangle$, then $A$ is isomorphic to 
$C(p, r-1)$. If $i\not = 0$, then $(ac^i )^p=c^{ip}$ which gives that $A=\langle b, ac^i\rangle$ is a metacyclic group 
$M(p, r-1)$ isomorphic to a split extension $C_{p^{r-2}}\rtimes C_p$ 
(see \cite[Lemma A.8]{DRV}).

Let $S$ be a subgroup of $G$ such that $A\cap S=U$ and $S/U\cong SL_2(p)$. Note that $S$ acts transitively (by conjugation) on the set of maximal subgroups of $A$ containing $U$. If $A\cong C_p \times C_{p^{r-2}}$ or $A\cong C_{p^{r-2}} \rtimes C_p$, then $U=\Phi(A)$ and the maximal subgroups of $A$ are not isomorphic, so they cannot be permuted transitively by $S$.
Hence these cases are not possible. Now assume that $A=\langle b, c^p, a\rangle \cong C(p, r-1)$. In this case $A$ can be expressed as a central product $E \ast U$ where  $E=\Omega_1(A)$ is the subgroup generated by elements of order $p$, in fact generated by $a$ and $b$, and $U=\langle c^p \rangle$. Note that $E$ is an extra-speical $p$-group of order $p^3$ and $U$ is a cyclic group of order $p^{r-3}$. Since $\Omega _1(A)$ is characteristic in $A$, $E$ is normal in $G$.

We already observed above that $U \leq Z(G)$, hence $U$ is a central subgroup of $S$. Since $S/U \cong SL_2(p)$ is perfect and has a trivial Schur multiplier (see for example \cite[Chapter 5]{Huppert}), we have $S \cong U \times SL_2 (p)$. So $G$ has some subgroup $L\cong SL_2 (p)$.
Consider the semidirect product $LE\leq G$. Let $z$ denote a central element in $E$. Observe that $E/\langle z \rangle $ is a natural module for $L$, and $L$ centralizes $\langle z \rangle$. According to \cite[(3F)]{Winter} there is some $x\in L$ satisfying $a^x=ab$ and $b^x =b$. Then the order of $x$ is $p$ and 
$\langle x, b, z \rangle \cong (\mathbb{Z} /p) ^3 $ which contradicts with the rank assumption. We conclude that the case 
$P= G(p, r; \epsilon)$ with $r\geq 4$ cannot occur.
\end{proof}
 
\begin{proof}[Proof of Theorem \ref{thm:intro1}]
Note that if $G$ $p'$-involves $\Qdp$ then clearly $G$ involves $\Qdp$. Thus, we only show that if $G$ involves 
$\Qdp$ then $G$ $p'$-involves $\Qdp$. Let $G$ be a minimal counterexample to the claim. Suppose that 
$X/Y\cong \Qdp$ for $Y\lhd X<G$. Clearly, the Sylow $p$-subgroups of $X$ are of rank $2$. Then we obtain that 
$X$ $p'$-involves $\Qdp$ by the minimality of $G$, and hence so does $G$. This contradiction shows that if 
$G$ has a section $X/Y$ isomorphic to $\Qdp$, then $X=G$. In particular, there exists $H\lhd G$ such that 
$G/H=\Qdp$.

Now let $S\in Syl_p(H)$. We have $G=HN_G(S)$ by the Frattini argument. It follows that 
$\Qdp\cong G/H\cong N_G(S)/N_H(S)$, and so $N_G(S)=G$ by the argument in the first paragraph. 
In particular, $S$ is normal in $H$, and hence $H$ has a Hall $p'$-subgroup $K$ by the Schur-Zassenhaus 
theorem. Moreover, any two Hall $p'$-subgroup of $H$ are conjugate in $H$ by the conjugacy part of 
Schur-Zassenhaus theorem. Therefore, we obtain $G=HN_G(K)$ by \cite[Thm 5.51]{Machi}. 
This yields that $N_G(K)/N_H(K)\cong \Qdp$ in a similar way, and so $K\lhd G$.

We claim that $K=1$. Assume the contrary. Write $\overline G=G/K$. Note that 
$\overline G/\overline H\cong G/H=\Qdp$ and $\overline P\cong P$. It follows that $\overline G$ 
satisfies the hypotheses, and so $\overline G$ $p'$-involves $\Qdp$ due to the fact that $|\overline G |<|G|$. 
Then we have $\overline X/\overline Y\cong \Qdp$ where $\overline Y$ is a $p'$-subgroup. Hence, we observe 
that $Y$ is a $p'$-group as $K$ is a $p'$-group. However, this is not possible as $X/Y\cong \Qdp$. 
This contradiction shows that $K=1$, that is, $H$ is a $p$-subgroup of $G$. It follows that $G\cong \Qdp$ 
by Lemma \ref{main lemma}. This contradiction completes the proof.
\end{proof}

Theorem \ref{thm:intro1} is not true when $p=3$ as the following example illustrates.

\begin{example}\label{ex:p=3} First note that $SL(2, 3) \cong Q_8 \rtimes C_3$ where $Q_8=\{\pm 1, \pm i, \pm j , \pm k \}$ and $C_3$ permutes the elements $i, j, k$ cyclically. Let $A=\langle a \rangle$ be a cyclic group of order $9$ and $Q$ be the quaternion group of order $8$. 
Constitute $T$ as the semidirect product $Q \rtimes A$ such that  $C_A(Q)=\langle a^3 \rangle$ and $A/\langle a^3 \rangle \cong C_3$ 
acts on $Q$ as above. Then we have $T/ \langle a^3 \rangle \cong SL(2,3)$. Note that $T$ is a non-split extension of $SL(2, 3)$ 
by $C_3$, and there is no such extension for $p>3$.  

Let $E$ be a nonabelian group of order $27$ with exponent $3$. We can take a presentation for $E$ as follows:
$$E=\langle x,y, z \, | \, x^3=y^3=z^3=1, [x,y]=z, [x,z]=[y,z]=1 \rangle.$$ 
We have an embedding of $T/ \langle a^3 \rangle \cong SL(2, 3)$ into $Aut(E)$ which takes an element $$\begin{bmatrix} u & v \\ s & t \end{bmatrix} \in SL(2,3)$$ to the automorphism of $E$ defined by $x\to x^u y^s$ and $y \to x^v y^t$. Let $\varphi $ denote the composition $T \to T/\langle a^3 \rangle \to Aut(E)$. Define $G$ to be the semidirect product of $E$ by $T$  using the homomorphism $\varphi$.
Set $\overline G=G/\langle z^{-1}a^3 \rangle $. We shall show that $\overline G$ 
is the desired counterexample. 

We first study $G$. Let $P=EA$. Then $P$ is generated by $a,x,y$. 
We can assume $a$ maps to the element 
$$\begin{bmatrix} 1 & 1 \\ 0 & 1 \end{bmatrix} \in SL(2, 3),$$
otherwise we can replace $a$ with a conjugate of $a$ in $SL (2, 3)$. 
Hence we have the relations ${}^a x=x$ and ${}^a y=xy$ together with the relations coming from $E$.
Using these relations it is easy to see that the Frattini subgroup 
$\Phi (P)$ of $P$ is the subgroup generated by the elements $ z, x, a^3$. The quotient group 
$P/\Phi ( P) $ is isomorphic to $C_3 \times C_3$ and it is generated by $\overline{a}$ and $\overline{y}$.
If $M$ is a maximal subgroup of $P$, then there are 4 possibilities, namely $M$ is generated by $\Phi (P)$ and 
a single element in $P$ which can be taken as one of the elements  $a, ay, a^2y,$ or $y$. Since $[a, x]=1$, and $[y, x]\neq 1$,
in the last 3 cases the maximal subgroup $M$ is not abelian. Note that in these cases $\overline M=M/\langle z^{-1} a^3 \rangle $ is also nonabelian.
The only case where $M$ is an abelian group is when $M=\langle z,x, a \rangle$. In this case $M$ is isomorphic to $C_3\times C_3 \times C_9$,
and the quotient group $\overline M=M/\langle z^{-1} a^3 \rangle $ is isomorphic to $C_3\times C_9$.
	  
Now let $X$ be an elementary abelian $3$-subgroup of $\overline P= P/\langle z^{-1}a^3 \rangle$ with $\rk(X)=3$. 
Then $X$ is a maximal subgroup of $\overline P$, hence it is of the form $\overline M$ for some maximal subgroup 
$M$ of $P$. As we have shown above there are only 4 possibilities for $\overline M$, and it is either nonabelian or it is isomorphic to $C_3\times C_9$. This contradicts the fact that $X\cong C_3\times C_3 \times C_3$. Hence $\overline P$ has rank equal to $2$, therefore $rk_3(\overline G)=2$. It is clear that $\overline G/\langle \overline z \rangle \cong \Qd (3)$, hence $\overline G$ involves $\Qd (3)$. 

Suppose that $\overline G$ $3'$-involves $\Qd (3)$. Then as $|\overline G | =3|\Qd(3)|$, it follows that 
$\overline G$ has  a subgroup $H$ isomorphic to $\Qd (3)$. Since $\Qd (3)$ does not have a normal subgroup isomorphic to $C_3$, 
we have $H \cap \langle \overline z \rangle=1$, which contradicts the fact that $\overline G$ has $3$-rank two. Hence $\overline G$ is a rank 2 group which involves $\Qd(3)$ but does not $3'$-involve $\Qd(3)$.
\end{example}
 
For $p=2$ there is a similar example. 
 
\begin{example}\label{ex:p=2} The group $\Qd(2)$ is isomorphic to the Symmetric group $S_4$, and its Sylow $2$-subgroup
is isomorphic to the Dihedral group of order $8$. The mod $2$ cohomology of
$S_4$ is given by  $$H^* (S_4; \bF_2) \cong \bF_2 [x_1, y_2, c_3]/(x_1c_3),$$
where the restriction of the 2-dimensional class $y_2$ to elementary abelian subgroups $V_1$ amd $V_2$ are both nonzero (see \cite[Ex 4.4]{AdemSurvey}). This shows that if we take the central extension $$1 \to \bZ/2 \to G \to S_4 \to 1 $$ with extension class $y_2 \in H^2 (S_4; \bF_2)$,
 then $G$ cannot include an elementary abelian subgroup isomorphic to $(\bZ/ 2)^3 $, hence $\rk (G) =2$. It is also easy to see that 
 $G$ does not include $\Qd(2)\cong S_4$ as a subgroup.
 \end{example}
 
 
\section{Spherical fibrations and Glauberman's ZJ-theorem}\label{sect:Glauberman}
 
Let $S$ be a $p$-group, and $\cA$ denote the set of all abelian subgroups in $S$. Let $d=\max\{ |A| \, |\, A\in \cA\}$. 
The \emph{Thompson subgroup} $J(S)$ of $S$ is defined as the subgroup generated by all subgroups $A\in \cA$ with $|A|=d$. 
We denote the center of $J(S)$ by $ZJ(S)$. Glauberman's ZJ-theorem \cite[Thm B]{Glauberman1} says that if $p$ is an odd prime and  
$G$ does not involve $\Qd(p)$, then $N_G(ZJ(S))$ controls $G$-fusion in $S$. 

There are three different definitions of Thompson's subgroup. Similar to $J(S)$, one can define $J_e(S)$ to be the subgroup of $S$ generated by all elementary abelian $p$-subgroups of $S$ with maximal order.  Although it is commonly believed that Glauberman's ZJ-theorem \cite[Thm B]{Glauberman1} is valid also for $J_e (S)$, we were not
able to find a published reference for this. Because of this, we will be using an analog of Glauberman's theorem due to Stellmacher.

\begin{theorem}[Stellmacher's ZJ-theorem]\label{thm:SZJ}
Let $p$ be a prime (possibly $p=2$), and $G$ be a finite group with Sylow $p$-subgroup $S$. If $G$ does not involve $\Qdp$, then there exists a characteristic subgroup $W(S)$ of $S$ such that $$\Omega(Z(S))\leq W(S) \leq \Omega(Z(J_e(S)))$$ and $N_G(W(S))$ 
controls $G$-fusion in $S$.
\end{theorem}

Theorem \ref{thm:SZJ} is a natural consequence of Stellmacher's normal ZJ-theorem (see \cite[Thm 9.4.4]{Kurzweil} and \cite{Stellmacher}). 
However it is not easy to see how this implication works since the conditions and conclusions of these two theorems are stated differently. We explain below
how Theorem \ref{thm:SZJ} follows from \cite[Thm 9.4.4]{Kurzweil} for the convenience of the reader. First we need a definition.

\begin{definition}\cite[pg 22]{Glauberman2}\label{def:p-stable}
A group $G$ is called \emph{$p$-stable} if it satisfies the following condition: Whenever $P$ is a $p$-subgroup 
of $G$, $g\in N_G(P)$ and $[P,g,g]=1$ then the coset $gC_G(P)$ lies in $O_p(N_G(P)/C_G(P))$. 
\end{definition}

\begin{proof}[Proof of Theorem \ref{thm:SZJ}]
 By part $(a)$ and $(c)$ of \cite[Theorem 9.4.4]{Kurzweil}, we see that $W$ is a section conjugacy functor (see the definition in \cite[pg 15]{Glauberman2}). First assume that $p$ is odd. Since $G$ does not involve $\Qd(p)$, by \cite[Proposition 14.7]{Glauberman2}, every section of $G$ is $p$-stable (according to Definition \ref{def:p-stable}). It is easy to see that then every section of $G$ is also $p$-stable according to the definition in \cite[pg 255]{Kurzweil}. By part (b) of \cite[Thm 9.4.4]{Kurzweil}, this gives that if $H$ is a section of $G$ such that 
 $C_H(O_p(H))\leq O_p(H)$ and $S$ is a Sylow $p$-subgroup of $H$, then $W(S)$ is a normal subgroup of $H$.  Now Theorem \ref{thm:SZJ} follows from \cite[Thm 6.6]{Glauberman2}.

For $p=2$, Stellmacher's normal ZJ-theorem still holds under the condition that $G$ does not involve $S_4 \cong \Qd(2)$  (see the main theorem and the remark after that in \cite{Stellmacher}).   
Hence the $p=2$ case of Theorem \ref{thm:SZJ} follows from \cite{Stellmacher} and \cite[Thm 6.6]{Glauberman2}.
\end{proof}

A spherical fibration over the classifying space $BG$ is a fibration $\xi : E \to BG$ whose fibre
is homotopy equivalent to a sphere.  A mod-$p$ spherical fibration over $BG$ is a fibration whose fibre
is homotopy equivalent to a $p$-completed sphere $(S^n )_p ^{\wedge}$ for some $n$.  The Euler class of a mod-$p$ 
spherical fibration is a cohomology class $e(\xi)$ in $H^{n+1} (BG; \bF_p)$ defined using the Thom isomorphism (see \cite[Sec 6.6]{OkayYalcin} for details).  
A cohomology class $u \in H^i (BG; \bF_p)$ is called \emph{$p$-effective} if for every elementary abelian $p$-subgroup 
$E \leq G$ with $\rk (E)=\rk_p (G)$, the  image of the restriction map $\res^G _E (u)$  is non-nilpotent in the cohomology ring $H^{*} (E; \bF_p)$.  

One of the important steps in constructing free actions of a finite group $G$ on a product of spheres is to
construct a mod-$p$ spherical fibration over $BG$ with a $p$-effective Euler class. To construct a mod-$p$ spherical fibration with
a $p$-effective Euler class, it is enough to construct a spherical (or mod-$p$ spherical) fibration over the $p$-completed classifying space 
$BG_p ^\wedge$. This is because given a fibration over $BG_p ^\wedge$, we can first apply $p$-completion and then take the pull-back along the completion map $BG\to BG_p ^\wedge$ (see \cite[Sec 6.4]{OkayYalcin}).

\begin{theorem}[Jackson \cite{Jackson}]\label{thm:Jackson} 
Let $p$ be an odd prime and $G$ be a finite group with $\rk _p(G)=2$. If $G$ does not $p'$-involve $\Qdp$,
then there is a mod-$p$ spherical fibration over $BG$ with a $p$-effective Euler class. 
\end{theorem}

Jackson proves Theorem \ref{thm:Jackson} in two steps. Let $S$ be a Sylow $p$-subgroup of $G$ and $\chi$ a character of $S$.
We say $\chi$ \emph{respects fusion in $G$} if for every $x\in S$ and  $g\in G$ satisfying $gxg^{-1}\in S$, the equality $\chi (x)=\chi (gxg^{-1})$ holds.
A character of a group $H$ is \emph{$p$-effective} if for every elementary abelian $p$-subgroup $E\leq H$ with $\rk (E)=\rk_p(H)$, 
the restriction of $\chi$ to $E$ does not include the trivial character as a summand, i.e., if $[\chi |_E , 1_E]=0$. Jackson \cite[Thm 45]{Jackson} proves that if $G$ is a finite group satisfying the conditions of Theorem \ref{thm:Jackson}, then there is a $p$-effective character of $S$ that respects fusion in $G$. 

Jackson \cite[Thm 16]{Jackson}  also shows that if $S$ has a $p$-effective character respecting fusion in $G$, then there is a complex vector bundle $E \to BG_p ^\wedge$ with a $p$-effective Euler class. The sphere bundle of this vector bundle gives the desired spherical fibration over $BG_p ^\wedge$. The construction of the vector bundle with a $p$-effective Euler class uses a mod-$p$ homotopy decomposition of $BG$, and they are constructed by proving the vanishing of certain higher limits as obstruction classes.
 
We show below that there is a more direct construction of a spherical fibration with a $p$-effective Euler class using a characteristic subgroup of a Sylow $p$-subgroup in $G$ as in Theorem \ref{thm:SZJ}.

\begin{proposition}\label{pro:main2}
Let $p$ be a prime, and $G$ a finite group with a Sylow $p$-subgroup $S$. Suppose that there is
a nontrivial characteristic subgroup $Z$ of $S$ such that $Z \leq \Omega(ZJ_e(S))$ and that the normalizer  $N_G(Z)$ controls 
$G$-fusion in $S$. Then there is a mod-$p$ spherical fibration over $BG$ with a $p$-effective Euler class.	
\end{proposition}

\begin{proof} Let $N=N_G(Z)$. By the Cartan-Eilenberg theorem, the restriction map $H^*(G; \bF_p)\to H^*(N, \bF_p)$ is an isomorphism, hence the $p$-completion of the classifying space $BG_p ^{\wedge}$ is homotopy equivalent to $BN_p^\wedge$. Therefore it is enough to construct the desired mod-$p$ spherical fibration over $BN_p ^{\wedge}$. We show below that there is a representation $V$ of $N$ that is $p$-effective. Then, the Borel construction $$S(V) \to EN\times _N S(V) \to BN$$ for the unit sphere $S(V)$ gives a spherical fibration with a $p$-effective Euler class. We can also assume $N$ acts trivially on $H_i (S(V); \bF_p)$ by replacing $V$ with $V\oplus V$ if necessary. Taking the $p$-completion of this fibration gives the desired mod-$p$ spherical fibration. Note that to conclude that fibers are homotopy equivalent to $S(V)_p^\wedge$, we use the Mod-$R$ fiber lemma by Bousfield and Kan \cite[Thm 5.1, Chp VI]{BousfieldKan}.
 
Now we show how to construct $V$. If $E$ is an elementary abelian subgroup of $S$ of maximum rank, then $E \leq J_e(S)$, hence  $E$ and $Z$ commute. Since the rank of $E$ is the maximum possible rank in $G$, and $EZ$ is an elementary abelian $p$-group, we must have $Z \leq E$. 

Let $\rho : Z \to \mathbb{C} ^\times$ be a nontrivial one-dimensional representation of $Z$. Consider the induced representation $V=\ind _Z ^N \rho$. By Frobenius reciprocity, we have
$$[\res ^N _Z \ind _Z ^N \rho, 1]=[\rho, \res ^N _Z \ind ^N _Z 1].$$
As $Z$ is normal in $N$, $$ \res^N _Z \ind _Z ^N 1=[N:Z] 1.$$ So, $C_V(Z)=0$, and therefore $C_V (E)=0$ whenever $Z\leq E$. 
This shows that $V$ is a $p$-effective character of $N$, hence completes the proof.
\end{proof}

Now we are ready to prove Theorem \ref{thm:intro2}.

\begin{proof}[Proof of Theorem \ref{thm:intro2}] Let $G$ be a finite group that does not involve $\Qd(p)$. Then, by Theorem \ref{thm:SZJ}, there exists a subgroup $W(S)$ of $S$ satisfying the conditions of Proposition \ref{pro:main2}. Applying Proposition \ref{pro:main2} we obtain the desired fibration.
\end{proof}

\begin{remark} The argument above also shows that when $G$ is a finite group which does not involve $\Qdp$, then there is a $p$-effective representation
$\chi$ of a Sylow $p$-subgroup $S$ which respects fusion in $G$. Note that since in this case $N=N_G(W(S))$ controls $G$-fusion in $S$, it is enough to construct this representation so that it respects fusion in $N$. To achieve this we can take $\chi =\res^N_S V$ where $V$ is the representation of $N$ constructed above in the proof of Theorem \ref{pro:main2}.    
\end{remark}

In \cite[Thm 1.2]{OkayYalcin} it is shown that there is no mod-$p$ spherical fibration over $B\Qdp$ with $p$-effective Euler class when $p$ is an odd prime. 
This result can be extended to give a converse to Jackson's theorem.

\begin{theorem}[Jackson \cite{Jackson}, Okay-Yal\c c\i n \cite{OkayYalcin}]\label{thm:Iff}
Let $p$ be an odd prime and $G$ be a finite group with $\rk _p(G)=2$.  
Then $G$ does not $p'$-involve $\Qdp$ if and only if there is a mod-$p$ spherical fibration over 
$BG$ with effective Euler class.	
\end{theorem}

\begin{proof} One direction of this is Theorem \ref{thm:Jackson} stated above.
For the other direction, let $K\lhd H \leq G$ such that $H/K\cong \Qdp$ and $K$ has order coprime to $p$.
If there is a mod-$p$ spherical fibration $\xi_G$ over $BG$, we can pull it back to a mod-$p$ spherical fibration $\xi_H$ 
over $BH$ via the map $BH\to BG$ induced by inclusion. Since $\rk_p(H)=2$, if $\xi_G$ has a $p$-effective Euler class, then
$\xi_H$ also has a $p$-effective Euler class. 

By taking the $p$-completion of $\xi_H$, we get a mod-$p$ spherical fibration 
over $BH_p ^\wedge$ that has $p$-effective Euler class. Since $K$ is a subgroup with coprime order to $p$, the quotient map $H \to H/K$ induces a 
homotopy equivalence $BH_p^{\wedge} \cong B(H/K)_p^{\wedge}$. Hence we obtain a mod-$p$
spherical fibration over $B\Qdp _p ^\wedge$ with a $p$-effective Euler class.
But there is no such fibration by  \cite[Thm 1.2]{OkayYalcin}.
\end{proof}
 
For $p=2$, the condition that $G$ does not involve $\Qd(2) \cong S_4$ is not a necessary condition for the existence of a mod-$2$ 
spherical fibration over $BG$. The group $G=S_4$ acts on a 2-sphere $X=S^2$ with rank one isotropy, hence the Borel construction for this action
$X \to EG\times _G X \to BG$ gives a spherical fibration with a $p$-effective Euler class. It is interesting to ask whether the condition ``$G$ does not involve $\Qdp$" is necessary for the existence of a mod-$p$ spherical fibration over $BG$ with a $p$-effective Euler class when $G$ is group of arbitrary rank and $p$ is an odd prime.


\begin{thebibliography}{9}

\bibitem{AdemSurvey}
A.~Adem, \emph{Lectures on the cohomology of finite groups}. Interactions between homotopy 
theory and algebra, 317–334, Contemp. Math., \textbf{436}, Amer. Math. Soc., Providence, RI, 2007.

\bibitem{BousfieldKan}
A.~K.~Bousfield and D.~M.~Kan, \emph{Homotopy Limits, Completions and Localizations}, Springer-Verlag, 
1972.

\bibitem{Craven}
D. ~A.~Craven, \emph{The Theory of Fusion Systems}, Cambridge University Press, 2011. 

\bibitem{DRV}
A.~Diaz, A.~Ruiz, and A.~Viruel, \emph{All $p$-local finite groups of rank two for odd prime $p$}, 
Trans. Amer. Math. Soc. \textbf{359} (2007), 1725-1764.

\bibitem{Glauberman1}
G.~Glauberman, \emph{A characteristic subgroup of a $p$-stable group}, Canad. J. Math. 
\textbf{20} (1968), 1101-1135. 

\bibitem{Glauberman2}
G.~Glauberman, \emph{Global and local properties of finite groups}, In Finite simple groups 
(Proc. Instructional Conf., Oxford, 1969), pp. 1–64. Academic
Press, London, 1971. \textbf{20} (1968), 1101-1135.

\bibitem{Huppert}
B.~Huppert, \emph{Endliche Gruppen I}, Die Grundlehren der Mathematischen Wissenschaften, Band 134
Springer-Verlag, Berlin-New York 1967.
 
\bibitem{JacksonQuotient}
M.~A.~Jackson, \emph{A quotient of the set $[BG, BU(n)]$ for a finite group $G$ of small rank}, J. Pure and
Applied Alg. \textbf{188} (2004), 161-174. 		
		
\bibitem{Jackson}
M.~A.~Jackson, \emph{{${\rm Qd}(p)$}-free rank two finite groups act freely on
a homotopy product of two spheres}, J. Pure Appl. Algebra \textbf{208} (2007), 821--831.		

\bibitem{Kurzweil}
H.~Kurzweil and  B.~Stellmacher, \emph{The Theory of Finite Groups: An Introduction.} Springer-Verlag, New York, 2004.

\bibitem{Machi}	
A.~Machi, \emph{Groups: An Introduction to Ideas and Methods of the Theory of Groups.} 
Springer-Verlag, New York, 2004.

\bibitem{Stellmacher}
B.~Stellmacher, \emph{A characteristic subgroup for $\Sigma_4$-free groups}, 	
Israel J. Math. \textbf{94} (1996), 367-379. 

\bibitem{OkayYalcin}
C.~Okay and E.~Yal\c c\i n, \emph{Dimension functions for spherical fibrations}, Algebraic $\&$ Geometric Topology \textbf{18}
(2018), 3907-3941.

\bibitem{Winter}
D.~L.~Winter, \emph{The automorphism group of an extraspecial $p$-group}, 
Rocky Mountain J. Math. \textbf{2} (1972), 159-168.

\end{thebibliography}
\end{document}